\documentclass{amsart}

\usepackage{amsmath} 
\usepackage{amssymb}
\usepackage[T1]{fontenc}

\newcommand{\forces}{\Vdash}

\newcommand{\borel}{{\tt Borel}}
\newcommand{\ttspan}{{\tt span}}
\newcommand{\cf}{{\rm cf}}

\newcommand{\add}{{\rm add}}
\newcommand{\cof}{{\rm cof}}
\newcommand{\cov}{{\rm cov}}
\newcommand{\non}{{\rm non}}
\newcommand{\covt}{{\rm cov}^*}

\newcommand{\cB}{{\mathcal B}}

\newcommand{\cD}{{\mathcal D}}

\newcommand{\cF}{{\mathcal F}}
\newcommand{\cI}{{\mathcal I}}
\newcommand{\cJ}{{\mathcal J}}

\newcommand{\bK}{{\mathbf K}}
\newcommand{\cK}{{\mathcal K}}

\newcommand{\cN}{{\mathcal N}}
\newcommand{\cM}{{\mathcal M}}
\newcommand{\cP}{{\mathcal P}}

\newcommand{\bbQ}{{\mathbb Q}}

\newcommand{\mbR}{{\mathbb R}}
\newcommand{\cS}{{\mathcal S}}

\newcommand{\bX}{{\mathbf X}}

\newtheorem{theorem}{Theorem}[section]

\newtheorem{corollary}[theorem]{Corollary}

\theoremstyle{definition}
\newtheorem{problem}[theorem]{Problem} 
 
\newtheorem{definition}[theorem]{Definition}

\theoremstyle{remark}

\newtheorem{remark}[theorem]{Remark}

\begin{document}

\title{On Borel hull operations}

\author{Tomasz Filipczak}
\address{Institute of Mathematics\\ 
Technical University of \L \'{o}d\'{z}\\
ul. W\'{o}l\-cza\'{n}\-ska 215\\
93-005 \L \'{o}d\'{z}, Poland}
\email{tomasz.filipczak@p.lodz.pl}

\author{Andrzej Ros{\l}anowski}
\address{Department of Mathematics\\
 University of Nebraska at Omaha\\
 Omaha, NE 68182-0243, USA}
\email{roslanow@member.ams.org}
\urladdr{http://www.unomaha.edu/logic}

\author{Saharon Shelah}
\address{Institute of Mathematics\\
 The Hebrew University of Jerusalem\\
 91904 Jerusalem, Israel\\
 and  Department of Mathematics\\
 Rutgers University\\
 New Brunswick, NJ 08854, USA}
\email{shelah@math.huji.ac.il}
\urladdr{http://www.math.rutgers.edu/$\sim$shelah}

\thanks{The second and the third author acknowledge support from the United
  States-Israel Binational Science Foundation (Grant no.~2010405). 
  Publication 1031 of the third author.}  

\date{January 31, 2014}

\begin{abstract}
We show that some set-theoretic assumptions (for example Martin's Axiom)
imply that there is no translation invariant Borel hull operation on the
family of Lebesgue null sets and on the family of meager sets (in
$\mathbb{R}^{n}$). We also prove that if the meager ideal admits a monotone
Borel hull operation, then there is also a monotone Borel hull operation on
the $\sigma$--algebra of sets with the property of Baire.  
\end{abstract}

\maketitle

\section{Introduction}
On several occasions a property of subsets of the real line $\mbR$ is 
introduced by means of a cover or a representation of the given set in terms 
of other sets. For instance,
\begin{enumerate}
\item[(i)] a set $A\subseteq \mbR$ is meager if $A \subseteq
  \bigcup\limits_{n<\omega} A_n$ for some closed nowhere dense sets
  $A_0,A_1,A_2,\ldots \subseteq \mbR$;
\item[(ii)] a set $A\subseteq \mbR$ is said to be $\Sigma^0_\xi$ if 
  $A=\bigcup\limits_{n<\omega} A_n$ for some $A_0,A_1,A_2,\ldots \in
  \bigcup\limits_{\zeta<\xi} \Pi^0_\zeta$;
\item[(iii)] a set $A\subseteq \mbR$ is Lebesgue measurable if $A\subseteq B$ for
  some Borel set $B\subseteq \mbR$ such that $B\setminus A$ is Lebesgue
  negligible;
\item[(iv)] a set $A\subseteq \mbR$ has the Baire property if for some Borel set
  $B\subseteq \mbR$ we have $A\subseteq B$ and $B\setminus A$ is meager, etc. 
\end{enumerate}
It is natural to ask if the ``witnesses'' in the above definitions can be
chosen in a somewhat canonical or uniform way. For instance, we may wonder
if they can depend monotonically on the input set $A$ or if they can be
translation invariant. Thus, in the relation to definition (i), we
may ask if there are mappings $\varphi_0,\varphi_1, \varphi_2,\ldots$
defined on the ideal of meager subsets of $\mbR$, with values in the family
of all closed nowhere dense subsets of $\mbR$ such that 
\begin{itemize}
\item[$(\circledast)$] $A\subseteq \bigcup_{i<\omega} \varphi_i(A)$ (for
  each meager $A\subseteq \mbR$) 
\end{itemize}
and with the property that one of the following two demands is satisfied:
\begin{itemize}
\item[$(\circledast)^{\rm trans}$] for every meager set $A$ and a real
  number $r$ we have $\varphi_i(A+r)=\varphi_i(A)+r$ for all $i$, 
or
\item[$(\circledast)^{\rm monot}$] for all meager sets $A,B$ such that
  $A\subseteq B$ we have  $\varphi_i(A)\subseteq \varphi_i(B)$ for all $i$. 
\end{itemize}
Easily, neither of these is possible.  Suppose towards contradiction that
there are such mappings $\varphi_i$ (for $i<\omega$) satisfying
$(\circledast)+(\circledast)^{\rm trans}$. Then for each rational number
$q\in \bbQ$ we have $\varphi_i(\bbQ)=\varphi_i(\bbQ+q)=\varphi_i(\bbQ) +q$
and hence each $\varphi_i(\bbQ)$ is a closed nowhere dense set invariant
under rational translations and this is impossible. Let us argue that the
mappings $\varphi_i$ (for $i<\omega$) cannot be monotone. So suppose they
satisfy $(\circledast)+(\circledast)^{\rm monot}$. By induction on
$\alpha<\omega_1$ construct a sequence $\langle A_\alpha:\alpha<\omega_1
\rangle$ of meager sets so that
\[\bigcup_{\beta<\alpha}\bigcup_{i<\omega} \varphi_i(A_\beta) \subsetneq
  A_\alpha\quad \mbox{ for all }\alpha<\omega_1.\] 
Then for some $n$ the sequence $\langle \varphi_n(A_\alpha):\alpha<\omega_1
\rangle$ of closed sets has a strictly increasing (cofinal) subsequence, a
contradiction. 

A similar question associated with definition (ii) also has the negative
answer: M\'{a}trai and Zelen\'{y} \cite{MaZe09} showed that it is not
possible to choose monotone presentations for $\Sigma^0_\xi$ sets. However,
problems concerning the monotonicity of the choice of witnesses for (iii)
and (iv) cannot be decided within the standard set theory.  Elekes and
M\'{a}th\'{e} \cite{ElMa09} proved that the existence of monotone Borel
hulls for measurable sets is independent from ZFC, and parallel results for
the Baire property were given by Balcerzak and Filipczak \cite{BaFi11}.

\subsection*{Notation and basic definitions}
In the current note $\bX$ is a Polish space, $\borel$ denotes the family of
all Borel subsets of $\bX$, $\cM$ is the $\sigma$--ideal of all meager
subsets of $\bX$, and $\cN$ is the $\sigma$--ideal of all Lebesgue
negligible subsets of $\mbR^n$. The same notation $\cM,\borel$ will be used
in $\mbR^n$, too. 

Let $\cI$ be a $\sigma$--ideal of subsets of $\bX$. We say that a
  family $\cD\subseteq\cI$ is a {\em base\/} of $\cI$ if 
\[\big(\forall A\in \cI\big) \big( \exists B\in\cD\big) \big(A\subseteq
B\big).\]
We say that $\cI$ has a {\em Borel basis\/} if every set from $\cI$ can be
covered by a Borel set from the ideal $\cI$, i.e., $\borel\cap\cI$ is a
base of $\cI$. For such a $\sigma$--ideal $\cI$, let $\cS_\cI$ denote the
$\sigma$--algebra of subsets of $\bX$ generated by ${\rm Borel}\cup\cI$. 
Thus, in particular, the $\sigma$--algebra ${\tt Baire}$ of all sets with
the Baire property is ${\tt Baire}=\cS_{\cM}$.

Let $\cI\subseteq\cP(\bX)$ be a proper $\sigma$--ideal with a Borel base and
containing all finite subsets of $\bX$. Cardinal coefficients of $\cI$ are
defined as follows:  
\begin{align*}
\add\left( \cI\right) & :=\min\left\{\left\vert\cF\right\vert:\cF\subseteq
\cI,\ \bigcup \cF\notin \cI\right\} , \\
\cov\left( \cI\right) & :=\min\left\{\left\vert \cF\right\vert :\cF\subseteq
  \cI,\ \bigcup \cF= \bX\right\} , \\ 
\non\left( \cI\right) & :=\min \left\{\left\vert A\right\vert :A\subseteq X,\
  A\notin \cI\right\} , \\ 
\cof\left( \cI\right) & :=\min \left\{\left\vert \cF\right\vert :\cF\subseteq
  \cI,\ \left( \forall A\in \cI\right)\left(\exists B\in\cF\right) (A\subseteq
  B\right)\}.
\end{align*}
If $\bX=\mbR^n$, $r\in\mbR^n$ and $A,B\subseteq \mbR^n$, then we define
$A+r=\{a+r:a\in A\}$ and $A+B=\{a+b:a\in A,\ b\in B\}$. A family $\cF$ of
subsets of $\mbR^n$ is {\em translation invariant\/} if $A+r\in\cF$ for all
$A\in\cF$ and $r\in\mbR^n$. For a translation invariant ideal $\cI\subseteq
\cP(\mbR^n)$ we define the {\em transitive covering number\/} of $\cI$ as
\[\covt\left( \cI\right):=\min \left\{ \left\vert A \right\vert
  :A\subseteq \mbR^n,\ (\exists B\in \cI)(A+B=\mbR^n)\right\}.\]

For systematic study of the cardinal invariants mentioned above for the case
of $\cN$ and $\cM$ we refer the reader to Bartoszy\'{n}ski and Judah
\cite{BaJu95}.  

\begin{definition}
Let $\cI$ be a $\sigma$--ideal on $\bX$ with Borel basis and $\cF\subseteq
\cS_\cI$. 
\begin{enumerate}
\item A {\em Borel hull operation\/} on $\cF$ with respect to 
$\cI$ is a mapping $\psi:\cF\longrightarrow\borel$ such that $A\subseteq
\psi(A)$ and $\psi(A)\setminus A\in\cI$ for all $A\in\cF$.  
\item If the range of a Borel hull operation $\psi$ consists of sets of some
  Borel class $\cK$, then we say that $\psi$ is a {\em $\cK$ hull
    operation}. 
\item A Borel hull operation $\psi$ on $\cF$ is {\em monotone\/} if
  $\psi(A_1)\subseteq\psi(A_2)$ whenever $A_1\subseteq A_2$ are from $\cF$.  
\item Assume $\bX=\mbR^n$ and both $\cI$ and $\cF$ are translation
  invariant. If a Borel hull operation $\psi$ on $\cF$ satisfies $\psi(A+x)
  =\psi(A)+x$ for all $A\in\cF$ and $x\in\mathbb{R}^n$, then $\psi $ is
  called {\em a translation invariant hull operation}. 
\end{enumerate}
\end{definition}

By \cite{ElMa09,BaFi11}, under CH there exist monotone Borel hull 
operations on $\cS_{\cI}$ where $\cI$ denotes either the ideal of 
Lebesgue negligible sets or the meager ideal. Adding many random or 
Cohen reals to a model of CH gives a model with no monotone Borel hull 
operations for $\cI$ (where $\cI$ is either the null or the meager ideal, 
respectively). More examples of universes with and without monotone Borel 
hulls for the null and meager ideals were given in Ros{\l}anowski and 
Shelah \cite{RoSh:972}. 

\subsection*{The content of the paper}
In \cite[Question 4.2]{ElMa09}, the authors ask if it is possible to define
a Borel hull operation on $\cN$ which is monotone and translation
invariant. In the second section we show that some set-theoretic assumptions
(for example Martin's Axiom) imply that there is no translation invariant
Borel hull operation on $\cN$ and on $\cM$ (even without the requirement of
monotonicity).
\medskip

The non-existence of monotone Borel hull operations on $\cI$ implies 
non-existence of such operations on $\cS_\cI$ but not much had been known 
about the converse implication. In particular, Balcerzak and Filipczak 
\cite[Question 2.23]{BaFi11} asked if it is possible that there exists a 
monotone Borel hull operation on $\cI$ (with respect to $\cI$) but there 
is no such hull operation on $\cS_\cI$ for a ccc ideal $\cI$. In the third
section we give a negative answer for the case of the meager ideal. We show
that the existence of a monotone Borel hull operation on $\cM$ (with respect to 
$\cM$) is equivalent with the existence of such hull operation on 
$\cS_{\cM}$.

\section{No translation invariant Borel hulls operations}
It is known that pairs $\left( \cS_{\cN},\cN\right) $ and $\left(
  \cS_\cM,\cM\right)$ have the Extended Steinhaus Property, i.e., for any
$A,B\in \cS_\cN\setminus\cN$ ($\cS_\cM\setminus \cM$, resp.)  the set
$A-B=\left\{ a-b:a\in A,b\in B\right\}$ has an interior point. Many variants
of Steinhaus Property have been investigated in the literature (see
Bartoszewicz, Filipczak and Natkaniec \cite{BFN11}). We need a
generalization in which only one of the sets $A$, $B$ has to belong to the
$\sigma$--algebra. In \cite{Mc50} and \cite{BCS58} such properties were proved
for topological groups (locally compact groups with complete Haar measure,
resp.). We formulate them for $\mbR^n$. 

\begin{theorem}
\label{c}
\begin{enumerate}
\item \label{c1}\cite[Cor. 4]{Mc50} 
  If $A,B\subseteq \mbR^n$ are of the second category and $A$ has the Baire
  property, then $A+B$ has an interior point.
\item \label{c2}\cite[Thm 1]{BCS58} 
  If $A,B\subseteq \mbR^n$ are not Lebesgue null sets and $A$ is measurable,
  then $A+B$ has an interior point.
\end{enumerate}
\end{theorem}

To prove that there is no translation invariant Borel hull operation on
$\cN$ (on $\cM$, respectively) it is enough to show that the additive group  
$\mbR^n$ has a subgroup which is a null set of the second category (a meager
set of positive outer measure).

\begin{theorem}
\label{t6}
\begin{enumerate}
\item \label{t6-1}
  If $\mbR^n$ has a subgroup $G\in \cN\setminus \cM$, then there is no
  translation invariant Borel hull operation on $\cN$.
\item \label{t6-2}
  If $\mbR^n$ has a subgroup $H\in \cM\setminus \cN$, then there is no
  translation invariant Borel hull operation on $\cM$.
\end{enumerate}

\begin{proof}
(\ref{t6-1}) 
Suppose, contrary to our claim, that there is a translation invariant hull
operation $\varphi :\cN\longrightarrow \cN\cap\borel$. For every $x$ from
$G$ we have $G+x=G$, which gives $\varphi \left( G\right) +x=\varphi \left(
G+x\right) =\varphi \left( G\right) $, and consequently $\varphi \left(
G\right) +G=\varphi \left( G\right) $. Since $G\notin \cM$ and $\varphi
\left( G\right) \in \borel\setminus \cM$, Theorem \ref{c} implies that
$\varphi \left( G\right) +G$ has an interior point, contrary to $\varphi
\left( G\right) \in \cN$.  The proof of (\ref{t6-2}) is similar.
\end{proof}
\end{theorem}

We will show that under some set-theoretic assumptions one can find a linear
subspace of $\mbR^n$ (considered over the field $\bbQ$ of rational numbers)
which belongs to $\cN\setminus\cM$ ($\cM\setminus\cN$, respectively).

Let us consider inequalities $\non(\cN)>\non(\cM)$ and
$\non(\cM)>\non(\cN)$. From Cicho\'{n}'s diagram it follows that each of
them is independent in ZFC.  

\begin{theorem}[~S. G{\l}\k{a}b]
\label{t7}
\begin{enumerate}
\item \label{t7-1}
  If $\non(\cN)>\non(\cM)$, then there exists a linear subspace of $\mbR^n$
  which belongs to $\cN\setminus\cM$.
\item \label{t7-2}
  If $\non(\cM)>\non(\cN)$, then there exists a linear subspace of $\mbR^n$
  which belongs to $\cM\setminus \cN$.
\end{enumerate}

\begin{proof}
  (\ref{t7-1}) Let $A\subseteq\mbR^n$ be a set such that $\left\vert
    A\right\vert =\non\left(\cM\right) $ and $A\notin\cM$. From $\left\vert
    \ttspan(A)\right\vert =\left\vert A\right\vert $ we obtain $\ttspan(A)\in
  \cN\setminus \cM$. The proof of (\ref{t7-2}) is similar.
\end{proof}
\end{theorem}

A {\em $\kappa$--Luzin set\/} for an ideal $\cI$ on $\mbR^n$ is a subset of
$\mbR^n$ which has cardinality $\geq\kappa$, and such that the intersection
with any set from $\cI$ has cardinality less than $\kappa$ (compare
Bukovsk\'{y} \cite[Section 8.2]{Buk11} or Cicho\'{n} \cite{Ci89}). If
$\cI=\cN$, then $\kappa$--Luzin sets for $\cI$ are also called
$\kappa$--Sierpi\'{n}ski sets, and $\kappa$--Luzin sets for the meager ideal
are called just $\kappa$--Luzin. If $\kappa=\aleph_1$ then we may omit it.

Of course, a $\kappa$--Luzin set for $\cI$ does not belong to $\cI$. Note
that if there exists a $\kappa$--Luzin set for $\cI$, then $\non(\cI)\leq
\kappa$ and $\cf(\kappa)\leq\cov(\cI)$. It is known that if
$\cov(\cI)=\cof(\cI)=\kappa $, then there exists a $\kappa $--Luzin set for
$\cI$ (see \cite[Theorem 8.26]{Buk11}).  Since $\mbR^n$ can be decomposed
into a null set and a meager set, it follows that for any regular $\kappa $,
every $\kappa$--Luzin set has measure zero and every
$\kappa$--Sierpi\'{n}ski set is meager.

\begin{theorem}[{\protect\cite[Theorem 8.28]{Buk11}}]
\label{t9}
\begin{enumerate}
\item \label{t9-1}
  If $\kappa \leq\non(\cN)$ and $A$ is a $\kappa$--Luzin set, then $A\in
  \cN\setminus \cM$. In particular, if $\kappa$ is regular and $A$ is a
  $\kappa$--Luzin set, then $A\in\cN\setminus\cM$.
\item \label{t9-2}
  If $\kappa \leq\non(\cM)$ and $A$ is a $\kappa$--Sierpi\'{n}ski set, then
  $A\in\cM\setminus\cN$. In particular, if $\kappa$ is regular and $A$ is a
  $\kappa$--Sierpi\'{n}ski set, then $A\in\cM\setminus\cN$.
\end{enumerate}
\end{theorem}

Sm\'{\i}tal proved that if the Continuum Hypothesis holds then there exists
a linear subspace of $\mbR^n$, which is a Luzin set (see \cite{Sm68}). One
can construct a linear subspace which belongs to $\cN\setminus \cM$
($\cM\setminus \cN$) assuming a condition weaker than CH. The proof is a
small modification of the proof of \cite[Lemma 1]{Sm68}.

\begin{theorem}
\label{t8}
Let $\{\cI,\cJ\}=\{\cN,\cM\}$. 
\begin{enumerate}
\item If $\covt(\cI)\geq\cof(\cI)$, then there exists a linear subspace $H$
of $\mbR^n$ such that $H\in \cJ\setminus\cI$. 
\item {\rm [Bukovsk\'{y} \cite[Exercise 8.7(b)]{Buk11}]} If $\cov(\cI)
  =\cof(\cI)=\kappa$, then there exists a linear subspace $H$ of $\mbR^n$,
  which is a $\kappa$--Luzin set for $\cI$. 
\end{enumerate}
\end{theorem}

\begin{proof}
(1)\quad Let $\kappa=\cof(\cI)\leq \covt(\cI)$. Let $B\in \cI$ be such that
$\mbR^n\setminus B\in \cJ$ and let $\langle B_\alpha:\alpha< \kappa \rangle$
be a basis for the ideal $\cI$ such that $B\subseteq B_\alpha$ for all
$\alpha<\kappa$. By induction on $\alpha$ we choose a sequence $\langle
x_\alpha:\alpha<\kappa\rangle$. Suppose that $\langle
x_\beta:\beta<\alpha\rangle$ has been defined  and let $Z_\alpha:=
\ttspan(\{x_\beta:\beta<\alpha\})$. The set $\bbQ B_\alpha:=\{qx: x\in
B_\alpha,\ q\in\bbQ\}$ belongs to $\cI$ and $|Z_\alpha|<\covt(\cI)$, so 
we may choose 
\begin{enumerate}
\item[$(*)_\alpha$] $x_\alpha\in \mbR\setminus \bigcup\limits_{y\in
    Z_\alpha} (\bbQ B_\alpha+y)$.  
\end{enumerate}
Then, after the construction is carried out, we set 
\[H:=\ttspan(\{x_\alpha:\alpha<\kappa\}).\]
Since $x_\alpha\notin B_\alpha$, we have $H\notin\cI$. Also,
$H\setminus\{0\} \subseteq \mbR^n\setminus B$. Indeed, suppose towards
contradiction that $x\in H\cap B\setminus\{0\}$ and let $x=q_1x_{\alpha_1}+
\ldots+ q_mx_{\alpha_m}$, where $\alpha_1<\alpha_2< \ldots<\alpha_m<
\kappa$, $q_1,\ldots,q_m\in\bbQ$ and $q_m\neq 0$. Let $y=-\frac{1}{q_m} (q_1 
x_{\alpha_1}+\ldots+q_{m-1} x_{\alpha_{m-1}})$. Plainly, $y\in
Z_{\alpha_m}$. Since $x_{\alpha_m}\notin \bbQ B_{\alpha_m}+y$ and $x\in
B\subseteq B_{\alpha_m}$, we conclude 
\[x_{\alpha_m}\neq \frac{1}{q_m}x-\frac{1}{q_m}(q_1x_{\alpha_1}+
\ldots+q_{m-1}x_{\alpha_{m-1}})=x_{\alpha_m},\]
a contradiction. Now we easily see that $H\in \cJ\setminus\cI$. 
\medskip

(2)\quad The arguments are essentially the same as in (1). We choose a
sequence $\langle x_\alpha:\alpha<\kappa\rangle$ so that 
\begin{enumerate}
\item[$(*)_\alpha^+$] $x_\alpha\in \mbR\setminus \bigcup \big\{\bbQ B_\beta
  +y: \beta\leq \alpha\ \&\ y\in Z_\alpha\big\}$,  
\end{enumerate}
where $Z_\alpha:=\ttspan(\{x_\beta:\beta<\alpha\})$. After the construction
is carried out, we set  $H:=\ttspan(\{x_\alpha:\alpha<\kappa\})$,
and we argue that $\left\vert H\cap B_\alpha\right\vert<\kappa$ for each
$\alpha<\kappa$. Suppose $x\in H\cap B_\alpha$, $x\neq 0$. Then $x=q_1
x_{\alpha_1}+\ldots +q_mx_{\alpha_m}$ for some $q_1,\ldots,q_m\in \bbQ$,
$q_m\neq 0$ and $\alpha_1<\ldots<\alpha_m$. So,
\[y_0:=x_{\alpha_m}-\frac{x}{q_m}=-(\frac{q_1}{q_m}x_{\alpha_1}+ \ldots +
\frac{q_{m-1}}{q_m} x_{\alpha_{m-1}})\in\ttspan(\{x_\beta:\beta<\alpha_m\})
= Z_{\alpha_m}\] 
and 
\[x_{\alpha_m}=\frac{1}{q_m}x+y_0\in \bbQ B_\alpha+Z_{\alpha_m}.\]
Since, by $(*)^+_\alpha$, $x_{\alpha_m}\notin \bbQ B_\beta+Z_{\alpha_m}$ for
$\beta\leq \alpha_m$ we conclude $\alpha>\alpha_m$. Thus $H\cap
B_\alpha\subseteq Z_\alpha$, and consequently $\left\vert H\cap
  B_\alpha\right \vert<\kappa$.  
\end{proof}

\begin{remark}
Concerning the assumptions in Theorem \ref{t8}(1), note that
$\covt(\cN)\leq \non(\cM)\leq \cof(\cN)$, so in the case of the null
ideal the assumption here is actually $\covt(\cN)=\cof(\cN)$. However,
for the meager ideal we can only say that $\covt(\cM)\leq \non(\cN)$
and it is consistent that $\covt(\cM)>\cof(\cM)$. For further
discussion of $\covt(\cM)$ we refer the reader to Bartoszy\'{n}ski and
Judah \cite[Section 2.7]{BaJu95} or Miller and Stepr\={a}ns
\cite{MiSt06}. 

In Theorem \ref{t8}(2) note that $\cov(\cI)\leq \non(\cJ)$ and
therefore the subspace $H$ as there satisfies $H\in \cJ\setminus \cI$ 
(by Theorem \ref{t9}). 

It is well known that for $\cI \in \{\cN,\cM\}$, the Martin Axiom MA implies 
$\add(\cI) =\cov(\cI) =\covt(\cI) =\non(\cI)= \cof(\cI)=2^{\aleph_0}$. 
\end{remark}

The following corollary sums up our previous  considerations.

\begin{corollary}

\begin{enumerate}
\item There is no translation invariant Borel hull operation on $\cN$
if any of the following conditions holds: 

$\non(\cN)>\non(\cM) $ or $\covt(\cM)\geq \cof(\cM)$ or MA. 

\item There is no translation invariant Borel hull operation on $\cM$
if any of the following conditions holds:

$\non(\cM)>\non(\cN) $ or $\covt(\cN)=\cof(\cN)$ or MA. 
\end{enumerate}
\end{corollary}

\section{Monotone Borel hulls operations on {\tt Baire}}
Let us fix a countable base $\cB$ of our Polish space $\bX$. We also require
that $\cB$ is {\em closed under intersections\/} and $\bX\in \cB$. 

\begin{definition}
\label{regularhull}  
A monotone Borel hull operation $\varphi:\cM\longrightarrow {\tt Borel}
\cap \cM$ is {\sf $\cB$--regular\/} whenever  
\[\varphi(A)\cap U\subseteq \varphi(A\cap U)\]
for all $A\in \cM$ and $U\in \cB$. 
\end{definition}

\begin{theorem}
  \label{thmmeager}
The following conditions are equivalent:
\begin{enumerate}
\item[(i)] There is a monotone Borel hull operation on $\cM$ with respect 
to $\cM$.  
\item[(ii)] There is a $\cB$--regular monotone Borel hull operation on 
$\cM$ with respect to $\cM$. 
\item[(iii)] There is a monotone Borel hull operation on ${\tt Baire}$ with 
respect to $\cM$.   
\end{enumerate}
\end{theorem}

\begin{proof}
(i) $\Rightarrow$ (ii)\quad Let $\varphi:\cM\longrightarrow {\tt
  Borel}\cap \cM$ be a monotone Borel hull operation on $\cM$. For
$A\in \cM$ let 
\[\psi(A)=\bigcap\big\{(\bX\setminus U)\cup \varphi(A\cap U): U\in \cB
\big\}.\] 
Plainly, $\psi(A)$ is a Borel subset of $\bX$ (as a countable intersection of
Borel sets) and $\psi(A)\subseteq\varphi(A)$ (as $\bX\in \cB$). Also, if
$A\subseteq B\in \cM$ and $U\in \cB$, then $\varphi(A\cap U) \subseteq
\varphi(B\cap U)$ (as $\varphi$ is monotone) and hence $(\bX\setminus U)\cup
\varphi(A\cap U) \subseteq (\bX\setminus U)\cup\varphi(B\cap
U)$. Consequently, if $A\subseteq B\in \cM$, then $\psi(A)\subseteq
\psi(B)$. Therefore $\psi:\cM\longrightarrow {\tt Borel}\cap
\cM$ is a monotone Borel hull on $\cM$. To show that it is
$\cB$--regular suppose $A\in \cM$ and $U\in \cB$. 

Fix $V\in\cB$ for a moment and let $W=U\cap V$. Then $W\in \cB$ and 
\[\begin{array}{r}
\psi(A)\cap U\subseteq \big((\bX\setminus W)\cup \varphi(A\cap W)\big) \cap
U=\big((\bX\setminus (U\cap V))\cup \varphi(A\cap U\cap V)\big) \cap U
\subseteq\\
(U\setminus V)\cup \varphi(A\cap U\cap V)\subseteq 
(\bX\setminus V)\cup \varphi((A\cap U)\cap V).\ 
\end{array}\]
Thus $\psi(A)\cap U\subseteq (\bX\setminus V)\cup \varphi((A\cap U)\cap V)$
for all $V\in \cB$ and therefore $\psi(A)\cap U\subseteq \psi(A\cap U)$. 
\medskip

\noindent (ii) $\Rightarrow$ (iii)\quad Suppose that  $\varphi:\cM\longrightarrow
{\tt Borel}\cap \cM$ is a $\cB$--regular monotone Borel hull operation
on
$\cM$. For a set $Z\subseteq \bX$ let 
\[\bK(Z)=\bX\setminus\bigcup\big\{U\in\cB:U\cap Z\in\cM\big\}.\] 
Clearly, $\bK(Z)$ is a closed subset of $\bX$ and
\begin{enumerate}
\item[$(*)_1$] $Z\subseteq Y\subseteq\bX$ implies $\bK(Z)\subseteq \bK(Y)$,   
\item[$(*)_2$] $Z\setminus \bK(Z)\in \cM$, and 
\item[$(*)_3$] if $Z\subseteq \bX$ has the Baire property, then
  $\bK(Z)\setminus Z\in \cM$.
\end{enumerate}
For $Z\in {\tt Baire}$ let 
\[\psi(Z)=\bK(Z)\cup \varphi(Z\setminus\bK(Z)).\] 
Plainly, $\psi:{\tt Baire}\longrightarrow {\tt Borel}$ and for
$Z\in {\tt Baire}$:
\begin{enumerate}
\item[$(*)_4$] $Z\subseteq \bK(Z)\cup (Z\setminus \bK(Z))\subseteq
  \bK(Z)\cup \varphi(Z\setminus \bK(Z))=\psi(Z)$, and  
\item[$(*)_5$] $\psi(Z)\setminus Z\subseteq \big(\bK(Z)\setminus Z\big) \cup
  \varphi(Z\setminus \bK(Z))\in\cM$. 
\end{enumerate}
Thus $\psi$ is a Borel hull operation on ${\tt Baire}$. Let us argue
that $\psi$ is monotone, i.e., 
\begin{enumerate}
\item[$(*)_6$] if $Z\subseteq Y\subseteq \bX$, $Z,Y\in {\tt Baire}$,
  then $\psi(Z)\subseteq \psi(Y)$.  
\end{enumerate}
So suppose $Z\subseteq Y$ have the Baire property and $x\in \psi(Z)$. If
$x\in \bK(Y)$, then $x\in \psi(Y)$ by the definition of $\psi$. Thus assume
also that $x\notin \bK(Y)$. Hence $x\notin \bK(Z)$ (remember $(*)_1$) so
$x\in \varphi(Z\setminus \bK(Z))$. Let $U\in\cB$ be such that $x\in
U\subseteq \bX\setminus \bK(Y)$. Then $(Z\setminus \bK(Z))\cap U\subseteq
Y\setminus \bK(Y)$ and, since $\varphi$ is $\cB$--regular, 
\[x\in \varphi(Z\setminus\bK(Z))\cap U\subseteq \varphi((Z\setminus \bK(Z))
\cap U)\subseteq \varphi(Y\setminus \bK(Y))\subseteq \psi(Y).\] 
\medskip

\noindent (iii) $\Rightarrow$ (i)\quad Straightforward.
\end{proof}

The proof of Theorem \ref{thmmeager} also shows the following.

\begin{corollary}
\label{ourcor}
If there is a monotone $\Sigma^0_\xi$ ($\Pi^0_\xi$,  respectively) hull 
operation on $\cM$ with respect to $\cM$, $2\leq \xi<\omega_1$, then there 
exists a monotone $\Pi^0_{\xi+1}$ ($\Pi^0_\xi$, respectively) hull operation 
on {\tt Baire} with respect to $\cM$.   
\end{corollary}

Assuming CH, or more generally ${\rm add}(\cM)={\rm cof}(\cM)$, the 
$\sigma$--ideal of meager sets has a monotone $\Sigma^0_2$ hull operation 
and hence, under the same assumption, there is a monotone $\Pi^0_3$ hull 
operation on {\tt Baire}. Other situations with such hulls were provided 
in \cite{RoSh:972}.

\begin{definition}
[See {\cite[Definition 3.4]{RoSh:972}}]
\label{mhg}
Let $\cI$ be an ideal of subsets of $\bX$ and let $\alpha^*,\beta^*$ be
limit ordinals. {\em An $\alpha^*\times\beta^*$--base for $\cI$} is a
sequence $\langle B_{\alpha,\beta}:\alpha<\alpha^*\ \&\
\beta<\beta^*\rangle$ of Borel sets from $\cI$ such that 
\begin{enumerate}
\item[(a)] $\cB$ is a basis for $\cI$, i.e., $(\forall A\in\cI)(\exists B\in
  \cB)(A\subseteq B)$, and
\item[(b)] for each $\alpha_0,\alpha_1<\alpha^*$, $\beta_0,\beta_1<\beta^*$
  we have
\[B_{\alpha_0,\beta_0}\subseteq B_{\alpha_1,\beta_1}\quad \Leftrightarrow\quad
\alpha_0\leq \alpha_1\ \&\ \beta_0\leq\beta_1.\]
\end{enumerate}
\end{definition}

Note that if an ideal $\cI$ has an $\alpha^*\times\beta^*$--base, then 
${\rm add}(\cI)=\min\{{\rm cf}(\alpha^*),{\rm cf}(\beta^*)\}$
and ${\rm cof}(\cI)=\max\{{\rm cf}(\alpha^*),{\rm cf}(\beta^*)\}$.

\begin{theorem}
[See {\cite[Proposition 3.6]{RoSh:972}}]
Assume that $\alpha^*,\beta^*$ are limit ordinals. If an ideal $\cI$ has an 
$\alpha^*\times\beta^*$--base consisting of $\Pi^0_\xi$ sets, $\xi<\omega_1$,
then there exists a monotone $\Pi^0_\xi$ hull operation on $\cI$ with respect
to $\cI$.
\end{theorem}

\begin{theorem}
[See {\cite[Theorem 3.7]{RoSh:972}}]
\label{modelwithmhg}
Let $\kappa,\lambda$ be cardinals of uncountable cofinality, $\kappa\leq
\lambda$. There is a ccc forcing notion $\bbQ^{\kappa,\lambda}$ of size
$\lambda^{\aleph_0}$ such that
\[\forces_{\bbQ^{\kappa,\lambda}}\mbox{`` the meager ideal $\cM$ has a
  $\kappa\times\lambda$--basis consisting of $\Sigma^0_2$ sets ''.}
\] 
\end{theorem}

Putting the results quoted above together with Corollary \ref{ourcor}
we obtain the following.

\begin{corollary}
\label{corabove}
Let $\kappa,\lambda$ be uncountable regular cardinals, $\kappa\leq
\lambda$. There is a ccc forcing notion $\bbQ^{\kappa,\lambda}$ of size
$\lambda^{\aleph_0}$ such that   
\[\begin{array}{r}
\forces_{\bbQ^{\kappa,\lambda}}\mbox{`` there is a monotone $\Pi^0_3$ 
hull operation on {\tt Baire} and}\\
{\rm add}(\cM)=\kappa \mbox{ and }{\rm cof}(\cM)=\lambda\mbox{ ''.}
\end{array}\]
\end{corollary}

\section{Open problems}

\begin{problem}
\label{prob1}
\begin{enumerate}
\item Can we find, in ZFC, a subgroup of $\mbR^n$ which belongs to
  $\cN\setminus \cM$ ($\cM\setminus \cN$, respectively) ?
\item Which assumptions are necessary to find such subgroups in a locally 
compact group (with complete Haar measure)?
\end{enumerate}
\end{problem}

If the answer to Problem \ref{prob1}(1) is positive, then in ZFC there is no
translation invariant Borel hull on $\cN$ ($\cM$, respectively). Should the
existence of such subgroups be independent from ZFC, we still may suspect
that there are no translation invariant Borel hulls, or at least that
there are no translation invariant monotone Borel hulls. 
  
\begin{problem}
Is it consistent that there are translation invariant Borel hulls on $\cN$ ($\cM$,
respectively)?  If yes, can we additionally have that this hull
operation is monotone?
\end{problem}

The following problem is motivated by Theorem \ref{t8}(2). 

\begin{problem}
Let $\cI\in \{\cN,\cM\}$. Assume that there exists a $\kappa$--Luzin set for
$\cI$. Does there exists a subgroup of $\mbR^n$ which is also a
$\kappa$--Luzin set for $\cI$ ?
\end{problem}

Every set with Baire property has a $\Sigma^0_2$ hull, so one may wonder 
if in Corollary \ref{corabove} we may claim the existence of monotone 
$\Sigma^0_2$ hulls. Or even:

\begin{problem}
Is it consistent that there is a monotone $\Pi^0_3$ hull operation on 
{\tt Baire} but there is no monotone $\Sigma^0_2$ hull operation on 
{\tt Baire} (with respect to $\cM$) ?
\end{problem}

We do not know if an analogue of Theorem \ref{thmmeager} holds for the 
null ideal. 

\begin{problem}
[Cf. Balcerzak and Filipczak {\cite[Question 2.23]{BaFi11}}]
Is it consistent that there exists a monotone Borel hull on the ideal 
$\cN$ of Lebesgue negligible subsets of $\mbR$ (with respect to $\cN$) 
but there is no such hull on the algebra $\cS_\cN$ of Lebesgue 
measurable sets? In particular, is it consistent that ${\rm add}(\cN)=
{\rm cof}(\cN)$ but there is no monotone Borel hull operation on 
$\cS_\cN$ ?
\end{problem}


\end{document}